\documentclass{amsproc}
\usepackage{eurosym}
\usepackage{amssymb}
\usepackage{amsfonts}
\usepackage{cmtt}

\theoremstyle{plain}
\numberwithin{equation}{section}

\newtheorem{theorem}{Theorem}[section]
\newtheorem*{theoremm}{Theorem}
\newtheorem{lemma}[theorem]{Lemma}

\newtheorem{question}[theorem]{Question}
\newtheorem{conjecture}[theorem]{Conjecture}

\newtheorem{example}[theorem]{Example}

\DeclareMathOperator{\trdeg}{trdeg}
\newcommand{\field}[1]{\mathbb{#1}}
\newcommand{\N}{\field{N}}
\newcommand{\R}{\field{R}}
\newcommand{\Q}{\field{Q}}
\newcommand{\K}{\field{K}}

\begin{document}
\title[Splitting multidimensional necklaces]{Obstacles for splitting multidimensional necklaces}

\author[Micha\l\ Laso\'{n}]{Micha\l\ Laso\'{n}}

\dedicatory{\upshape
\'{E}cole Polytechnique F\'{e}d\'{e}rale de Lausanne, Chair of Combinatorial Geometry, EPFL-SB-MATHGEOM/DCG, Station 8, CH-1015 Lausanne, Switzerland\\ \textmtt{michal.lason@epfl.ch},\\
Institute of Mathematics of the Polish Academy of Sciences,\\ ul.\'{S}niadeckich 8, 00-656 Warszawa, Poland\\ \textmtt{michalason@gmail.com}}

%\author[Micha\l\ Laso\'{n}]{Micha\l\ Laso\'{n}\footnote{michal.lason@epfl.ch; \'{E}cole Polytechnique F\'{e}d\'{e}rale de Lausanne, Chair of Combinatorial Geometry, EPFL-SB-MATHGEOM/DCG, Station 8, CH-1015 Lausanne, Switzerland}
%\footnote{michalason@gmail.com; Institute of Mathematics of the Polish Academy of Sciences, ul.\'{S}niadeckich 8, 00-956 Warszawa, Poland}}
\thanks{Research supported by Polish National Science Centre grant no. 2012/05/D/ST1/01063 and by Swiss National Science Foundation Grants 200020-144531 and 200021-137574.}

\begin{abstract}
The well-known ``necklace splitting theorem'' of Alon \cite{Al87} asserts that every $k$-colored necklace can be fairly split into $q$ parts using at most $t$ cuts, provided $k(q-1)\leq t$. In a joint paper with Alon et al. \cite{AlGrLaMi09} we studied a kind of opposite question. Namely, for which values of $k$ and $t$ there is a measurable $k$-coloring of the real line such that no interval has a fair splitting into $2$ parts with at most $t$ cuts? We proved that $k>t+2$ is a sufficient condition (while $k>t$ is necessary).

We generalize this result to Euclidean spaces of arbitrary dimension $d$, and to arbitrary number of parts $q$. We prove that if $k(q-1)>t+d+q-1$, then there is a measurable $k$-coloring of $\R^d$ such that no axis-aligned cube has a fair $q$-splitting using at most $t$ axis-aligned hyperplane cuts. Our bound is of the same order as a necessary condition $k(q-1)>t$ implied by \cite{Al87}. Moreover for $d=1,q=2$ we get exactly the result of \cite{AlGrLaMi09}. 

Additionally, we prove that if a stronger inequality $k(q-1)>dt+d+q-1$ is satisfied, then there is a measurable $k$-coloring of $\R^d$ with no axis-aligned cube having a fair $q$-splitting using at most $t$ arbitrary hyperplane cuts. The proofs are based on the topological Baire category theorem and use algebraic independence over suitably chosen fields.
\end{abstract}

\maketitle

%%%%%%%%%%%%%%%%%%%%%%%%%%%%%%%%%%%%%%%%%%%%%%%%%%%%%%%%%%%%%%%%%%%%%%%%%%%%%%%%%%%%%%%%%%%%%%%%%%%%%%%%%%%%%%%%%%%%%%%%%%%%%%%%%%%%%%%%%%%%%%%%%%%%%%%%%%%%%%%%%%%%%%%%%%
\section{Introduction}
%%%%%%%%%%%%%%%%%%%%%%%%%%%%%%%%%%%%%%%%%%%%%%%%%%%%%%%%%%%%%%%%%%%%%%%%%%%%%%%%%%%%%%%%%%%%%%%%%%%%%%%%%%%%%%%%%%%%%%%%%%%%%%%%%%%%%%%%%%%%%%%%%%%%%%%%%%%%%%%%%%%%%%%%%%

The following problem is well-known as the necklace splitting problem: \smallskip

\emph{A necklace colored with $k$ colors has been stolen by $q$ thieves. They intend to share the necklace fairly, but they do not know what are the values of different colors. Therefore they want to cut it so that the resulting pieces can be fairly split into $q$ parts, which means that each part captures the same amount of every color. What is the minimum number of cuts they have to make?}\smallskip 

To be more precise a \emph{discrete $k$-colored necklace} is an interval in $\N$ colored with $k$ colors. A \emph{fair $q$-splitting of size $t$} is a division of the necklace using $t$ cuts into $t+1$ intervals, which can be split into $q$ parts with equal number of integers of every color. 

Goldberg and West \cite{GoWe85} proved that every discrete $k$-colored necklace with number of integers of every color divisible by $2$ has a fair $2$-splitting of size at most $k$ (see also \cite{AlWe86} for a short proof using the Borsuk-Ulam theorem, and \cite{Ma03} for other applications of the Borsuk-Ulam theorem in combinatorics). It is easy to see that it is best possible. Just consider a necklace consisting of $k$ pairs of consecutive integers, each pair colored with a different color. 

A famous generalization due to Alon \cite{Al87} to the case of $q$ thieves, asserts that every discrete $k$-colored necklace with number of integers of every color divisible by $q$ has a fair $q$-splitting of size at most $k(q-1)$. The same argument as before shows that the number of cuts is optimal.

All proofs of these discrete theorems go through a continuous version of the problem, in which a necklace is an interval, a coloring is a partition of a necklace into Lebesgue measurable sets, and a splitting is fair if each part captures the same measure of every color. 

Alon \cite{Al87} got even more general version in which a coloring is replaced by a collection of arbitrary continuous probability measures (this generalizes a theorem of Hobby and Rice \cite{HoRi65}). 

\begin{theorem}[Alon, \cite{Al87}]\label{al}
Let $\mu_1,\dots,\mu_k$ be continuous probability measures on the interval $[0,1]$. Then there exists a division of  $[0,1]$ using at most $k(q-1)$ cuts into subintervals, which can be split into $q$ parts having equal measure $\mu_i$ for each $i$. 
\end{theorem}

In this paper we are interested in multidimensional necklace splitting problem. Notice that every axis-aligned hyperplane in $\R^d$ is perpendicular to exactly one of axes and parallel to others. Those perpendicular to $i$-th axis we call \emph{aligned to $i$-th axis}. Hence any collection of $t$ axis-aligned hyperplane cuts gives a natural partition $t=t_1+\dots+t_d$, where exactly $t_i$ hyperplanes are aligned to $i$-th axis. By projecting, Alon's Theorem \ref{al} implies that for any $k$ continuous probability measures on $[0,1]^d$ there exists a fair $q$-splitting using at most $k(q-1)$ hyperplanes aligned to $1$-th axis. De Longueville and \v{Z}ivaljevi\'{c} generalized it to a version in which we can fix arbitrary partition $k(q-1)=t_1+\dots+t_d$. 

\begin{theorem}[de Longueville, \v{Z}ivaljevi\'{c}, \cite{LoZi08}]\label{lozi}
Let $\mu_1,\dots,\mu_k$ be continuous probability measures on $d$-di\-men\-sion\-al cube $[0,1]^d$. For any selection of non-negative integers $t_1,\dots,t_d$ satisfying $k(q-1)=t_1+\dots+t_d$, there exists a division of $[0,1]^d$ using only axis-aligned hyperplane cuts and at most $t_i$ aligned to $i$-th axis into cuboids, which can be split into $q$ parts having equal measure $\mu_i$ for each $i$. 
\end{theorem}

For the case of arbitrary continuous probability measures it is also easy to see that the total number of hyperplanes which suffice, that is $k(q-1)$, is best possible. Indeed, just divide the diagonal of the cube into $k$ intervals and consider one dimensional measures on them. However, this argument has still one dimensional nature.

In this paper we restrict our attention from continuous probability measures to measurable colorings, that is we are interested in splitting fully colored $d$-dimensional necklaces. A \emph{$k$-colored $d$-dimensional necklace} is a nontrivial axis-aligned cube in $\R^d$ partitioned into $k$ Lebesgue measurable sets, which we call colors. A \emph{fair $q$-splitting of size $t$} of a necklace is a division using $t$ axis-aligned hyperplane cuts into cuboids, which can be split into $q$ parts, each capturing exactly $1/q$ of the total measure of every color. 

A very interesting question concerning splitting multidimensional necklaces is to find the minimum number $t_{d,k,q}$ of cuts such that any $k$-colored $d$-dimensional necklace has a fair $q$-splitting using at most $t_{d,k,q}$ axis-aligned hyperplane cuts. Clearly, Alon's Theorem \ref{al} (or more generally Theorem \ref{lozi} of de Longueville and \v{Z}ivaljevi\'{c}) implies that $k(q-1)\geq t_{d,k,q}$. In Theorem \ref{TheoremNecklace} we actually prove that for $q=2$ this bound is tight. Together with Theorem \ref{lozi}  this fully characterizes the set of tuples $(t_1,\dots,t_d)\in\N^d$, such that any $k$-colored $d$-dimensional necklace has a fair $2$-splitting using at most $t_i$ hyperplanes aligned to $i$-th axis for each $i$. Namely, it is the set $$\{(t_1,\dots,t_d)\in\N^d:\;k\leq t_1+\dots+t_d\}.$$

However, the most interesting problem is a bit more difficult. We want to find a measurable coloring of the whole space $\R^d$, which avoids necklaces having a fair $q$-splitting of a bounded size $t$. Of course, we want to minimize the number of colors $k$ we use. This problem was already formulated and partially solved in \cite{AlGrLaMi09}. For $d=1,q=2$ a theorem of Goldberg and West \cite{GoWe85} (or more generally Alon's Theorem \ref{al}) implies that $k>t$ colors are needed, and we prove in \cite{AlGrLaMi09} that $k>t+2$ colors suffice.

\begin{theorem}[Alon, Grytczuk, Laso\'{n}, Micha\l ek, \cite{AlGrLaMi09}]\label{TheoremAGLM}
For every integers $k,t\geq 1$, if $k>t+2$, then there exists a measurable $k$-coloring of $\R$ such that no necklace has a fair $2$-splitting using at most $t$ cuts. 
\end{theorem}

Grytczuk and Lubawski \cite{GrLu12} generalized our result to higher dimensional spaces, by proving that the same holds in $\R^d$ (for $q=2$), provided  
$$k\geq\left(\frac{t}{d}+4\right)^d-\left(\frac{t}{d}+3\right)^d+\left(\frac{t}{d}+2\right)^d-2^d+t+2d+3.$$
They expected the lower bound on the number of colors is far from optimal, as even for $d=1$ it gives much worse result than the one from \cite{AlGrLaMi09}. Grytczuk and Lubawski asked \cite[Problem 9]{GrLu12} if $k\geq t+O(d)$ colors suffice for the case of two parts ($q=2$). In Section \ref{SectionAxisAlignedHyperplaneCuts} we prove that the answer is yes.

\begin{theoremm}\emph{\ref{TheoremNecklaces}}
For every integers $d,k,t\geq 1$ and $q\geq 2$, if $k(q-1)>t+d+q-1$, then there exists a measurable $k$-coloring of $\R^d$ such that no $d$-dimensional necklace has a fair $q$-splitting using at most $t$ axis-aligned hyperplane cuts. The set of such colorings is dense. For $d=1$ the same holds provided $k(q-1)>t+2$. 
\end{theoremm}

Notice that our bound is of the same order as a necessary condition $k(q-1)>t$, which is a consequence of Alon's Theorem \ref{al} (or by a more general Theorem \ref{lozi} of de Longueville and \v{Z}ivaljevi\'{c}). For $d=1,q=2$ Theorem \ref{TheoremNecklaces} coincides with the result of \cite{AlGrLaMi09}. We conjecture that for $q=2$ our bound is tight, as it captures the idea of degrees of freedom of objects we want to split.

In Section \ref{SectionArbitraryHyperplaneCuts} we generalize our results to the situation when arbitrary hyperplane cuts are allowed. This is closely related to the mass equipartitioning problem: \smallskip

\emph{Does for every $k$ continuous mass distributions $\mu_1,\dots,\mu_k$ in $\R^d$ there exist $t$ hyperplanes dividing the space into pieces which can be split into $q$ parts of equal measure $\mu_i$ for each $i$?}\smallskip

A necessary condition for the positive answer is $k(q-1)\leq td$. It can be shown by considering one dimensional measure along the moment curve. The most basic and well-known positive result in this area is the Ham Sandwich Theorem for measures, asserting that for $(d,k,q,t)=(d,d,2,1)$ the answer is yes. It is a consequence of the Borsuk-Ulam theorem (see \cite{Ma03}), and it also has discrete versions concerning equipartitions of a set of points in the space by a hyperplane. Several other things are known (see \cite{BlZi07,Ed87,MaVrZi06,Ma10,Ra96}), however even for some small values (eg. $(4,1,16,4)$) the question remains open. In the problem we are studying, $k$ continuous mass distributions in $\R^d$ are restricted to a measurable $k$-coloring of a given necklace. In this variant in Theorem \ref{TheoremNecklaceArbitrary} we reprove the necessary inequality for $q=2$. 

In the same spirit as in the case of axis-aligned hyperplane cuts we investigate existence of measurable colorings of $\R^d$ avoiding necklaces having a fair $q$-splitting of a bounded size $t$. 

\begin{theoremm}\emph{\ref{TheoremNecklacesArbitrary}}
For every integers $d,k,t\geq 1$ and $q\geq 2$, if $k(q-1)>dt+d+q-1$, then there is a measurable $k$-coloring of $\R^d$ such that no $d$-dimensional necklace has a fair $q$-splitting using at most $t$ arbitrary hyperplane cuts. The set of such colorings is dense. 
\end{theoremm}

In Section \ref{SectionColoringsDistinguishingCubes} we generalize another result from \cite{AlGrLaMi09} of a slightly different flavour. In \cite{AlGrLaMi09} we proved that there is a measurable $5$-coloring of the real line distinguishing intervals, that is no two intervals contain the same measure of every color (this solves a problem from \cite{GrSl03}). In Theorem \ref{2d+3} we generalize it by proving that there is a measurable $(2d+3)$-coloring of $\R^d$ distinguishing axis-aligned cubes. We suspect that this number is optimal. Recently Vre\'{c}ica and \v{Z}ivaljevi\'{c} \cite{VrZi13} solved the case of arbitrary continuous probability measures. They showed $d+1$ measures distinguishing cubes, and proved that for any $d$ continuous probability measures on $\R^d$ there are two (in fact any finite number) non-trivial axis-aligned cubes which are not distinguished.

The proofs of all already mentioned results are based on the topological Baire category theorem applied to the space of all measurable colorings of $\R^d$, equipped with a suitable metric. Section \ref{SectionTheSetting} describes this background. The remaining part of the argument is algebraic, it uses algebraic independence over suitably chosen fields. This algebraic independence describes well the idea of degrees of freedom of considered geometric objects. We will demonstrate in full details a proof of our main result -- Theorem \ref{TheoremNecklaces}. Proofs of other results fit the same pattern, so we will omit parts which can be rewritten with just minor changes and describe only main differences. 

We end the paper with a brief look at a discrete variant of the multidimensional necklace splitting problem, as introduced in \cite{GrLu12}. It turns out that the equivalence between discrete and continuous case does not hold in dimensions higher than one. In general it is more difficult to split a discrete necklace, one needs roughly $d$ times more axis-aligned hyperplane cuts. In Section \ref{SectionDiscreteCase} we give bounds on the minimum number of axis-aligned hyperplane cuts $t$ that are needed for a fair $q$-splitting of a $k$-colored discrete $d$-dimensional necklace (a cube in $\N^d$). Namely we prove that $d(k-1)\frac{q}{2}\leq t\leq (2d-1)k(q-1)$.

%%%%%%%%%%%%%%%%%%%%%%%%%%%%%%%%%%%%%%%%%%%%%%%%%%%%%%%%%%%%%%%%%%%%%%%%%%%%%%%%%%%%%%%%%%%%%%%%%%%%%%%%%%%%%%%%%%%%%%%%%%%%%%%%%%%%%%%%%%%%%%%%%%%%%%%%%%%%%%%%%%%%%%%%%%
\section{The Topological Setting}\label{SectionTheSetting}
%%%%%%%%%%%%%%%%%%%%%%%%%%%%%%%%%%%%%%%%%%%%%%%%%%%%%%%%%%%%%%%%%%%%%%%%%%%%%%%%%%%%%%%%%%%%%%%%%%%%%%%%%%%%%%%%%%%%%%%%%%%%%%%%%%%%%%%%%%%%%%%%%%%%%%%%%%%%%%%%%%%%%%%%%%

Recall that a set in a metric space is \emph{nowhere dense} if the interior of its closure is empty. Sets which can be represented as a countable union of nowhere dense sets are said to be of \emph{first category}.

\begin{theorem}[Baire, cf. \cite{Ox80}]
If $X$ is a complete metric space and a set $A\subset X$ is of first category, then $X\setminus A$ is dense in $X$ (in particular is non-empty).
\end{theorem}

Almost all of our results (except ones from the last Section) assert existence of certain `good' colorings. The main framework of the argument will mimic the one from \cite{AlGrLaMi09}. Our plan is to construct a suitable metric space of colorings and prove that the subset of `bad' colorings is of first category, resulting that `good' colorings exist. To prove it we will use entirely different technique than in \cite{AlGrLaMi09}. Our approach will be purely algebraic, we will compare transcendence degrees of some sets of numbers over suitably chosen fields. We go towards a proof of our main result -- Theorem \ref{TheoremNecklaces}, which asserts existence of a measurable $k$-coloring of $\R^d$ with no necklace having a fair $q$-splitting of size at most $t$, provided $k(q-1)>t+d+q-1$. 

Let $k$ be a fixed positive integer, and let $\{1,\dots,k\}$ be the set of colors. Let $M$ be a space of all measurable $k$-colorings of $\R^d$, that is the set of measurable maps $f:\R^d\rightarrow\{1,\dots,k\}$ (for each $i$ the set $f^{-1}(i)$ is Lebesgue measurable). For $f,g\in M$, and a positive integer $n$ let 
$$D_n(f,g)=\{x\in[-n,n]^d: f(x)\neq g(x)\}.$$
Clearly $D_n(f, g)$ is a measurable set, and we may define the normalized distance between $f$ and $g$ on $[-n,n]^d$ by 
$$d_n(f,g)=\frac{\mathcal{L}(D_n(f,g))}{(2n)^d},$$
where $\mathcal{L}(D)$ denotes the Lebesgue measure of a set $D$. We define the distance between any two colorings $f,g$ from $M$ by
$$d(f,g)=\sum_{n=1}^{\infty}\frac{d_n(f,g)}{2^n}.$$ 
To get a metric space we need to identify colorings whose distance is zero (they differ only on a set of measure zero), but they are indistinguishable from our point of view. Formally $\mathcal{M}$ is metric space consisting of equivalence classes of $M$ with distance function $d$. 

\begin{lemma}
Metric space $\mathcal{M}$ is complete.
\end{lemma}

This is a straightforward generalization of the fact that sets of finite measure in any measure space form a complete metric space, with measure of symmetric difference as the distance function (see \cite{Ox80}).

A coloring $f\in\mathcal{M}$ is called a \emph{cube coloring on $[-n,n]^d$} if there exists $N$ such that $f$ is constant on any half open cube from the division of $[-n,n]^d$ into $N^d$ equal size cubes. Let $\mathcal{I}_n$ denote the set of all colorings from $\mathcal{M}$ that are cube colorings on $[-n,n]^d$.

\begin{lemma}\label{aprox} Let $f\in\mathcal{M}$, then for every $\epsilon>0$ and an integer $n$ there exists a cube coloring $g\in\mathcal{I}_n$ such that $d(f,g)<\epsilon$.
\end{lemma}
\begin{proof} Let $C_i=f^{-1}(i)\cap[-n,n]^d$ and let $C_i^*\subset [-n, n]^d$ be a finite union of cuboids such that
$$\mathcal{L}((C_i\setminus C_i^*)\cup(C_i^*\setminus C_i))<\frac{\epsilon}{2k^2}$$
for each $i=1,\dots,k$. 
Define a coloring $h$ so that for each $i=1,\dots,k$, the set $C_i^*\setminus(C_1^*\cup\dots\cup C_{i-1}^*)$ is filled with color $i$, the remaining part of the cube $[-n,n]^d$ is filled with any fixed color, and $h$ agrees with $f$ everywhere outside $[-n,n]^d$. Then $d(f,h)<\epsilon/2$ and clearly each $h^{-1}(i)\cap[-n, n]^d$ is a finite union of cuboids. Let $A_1,\dots,A_t$ be the whole family of these cuboids. Now divide the cube $[-n, n]^d$ into $N^d$ equal cubes $B_1,\dots,B_{N^d}$ for $N>4dt(2n)^d/\epsilon$. Define a new coloring $g$ such that $g\vert_{B_i}\equiv h\vert_{B_i}$ whenever $B_i\subset A_j$ for some $j$, and $g\vert_{B_i}$ is constant equal to any color otherwise. Clearly $g$ is a cube coloring on $[-n,n]^d$. Moreover $g$ differs from $h$ on at most $2dt(N)^{d-1}$ small cubes of total volume at most $2dt(2n)^d/N<\epsilon/2$, thus we get $d(f,g)<\epsilon$.	
\end{proof}

%%%%%%%%%%%%%%%%%%%%%%%%%%%%%%%%%%%%%%%%%%%%%%%%%%%%%%%%%%%%%%%%%%%%%%%%%%%%%%%%%%%%%%%%%%%%%%%%%%%%%%%%%%%%%%%%%%%%%%%%%%%%%%%%%%%%%%%%%%%%%%%%%%%%%%%%%%%%%%%%%%%%%%%%%%
\section{Axis-aligned Hyperplane Cuts}\label{SectionAxisAlignedHyperplaneCuts}
%%%%%%%%%%%%%%%%%%%%%%%%%%%%%%%%%%%%%%%%%%%%%%%%%%%%%%%%%%%%%%%%%%%%%%%%%%%%%%%%%%%%%%%%%%%%%%%%%%%%%%%%%%%%%%%%%%%%%%%%%%%%%%%%%%%%%%%%%%%%%%%%%%%%%%%%%%%%%%%%%%%%%%%%%%

A splitting of a necklace $C=[\alpha_1,\alpha_1+\alpha_0]\times\dots\times[\alpha_d,\alpha_d+\alpha_0]$ using $t$ axis-aligned hyperplane cuts can be described by a partition $t_1+\dots+t_d=t$ such that for each $i$ there are exactly $t_i$ hyperplanes aligned to $i$-th axis, by $i$-th coordinates of this hyperplanes $\alpha_i=\beta^i_0\leq\beta^i_1\leq\dots\leq\beta^i_{t_i}\leq\beta^i_{t_i+1}=\alpha_i+\alpha_0$, and by a $q$-labeling of obtained pieces. By granularity of the splitting (or more generally of a division) we mean the length of the shortest subinterval $[\beta^i_j,\beta^i_{j+1}]$ over all $i,j$. 

For $(t_1,\dots,t_d)$ and $n$ we denote by $\mathcal{B}_n^{(t_i)}$ the set of $k$-colorings from $\mathcal{M}$ for which there exists at least one $d$-dimensional necklace contained in $[-n, n]^d$ which has a fair $q$-splitting using for each $i$ exactly $t_i$ hyperplanes aligned to $i$-th axis and with granularity at least $1/n$. Finally let $$\mathcal{B}_n=\bigcup_{t_1+\dots+t_d\leq t}\mathcal{B}_n^{(t_i)}$$
be a set of `bad' colorings. 

Let $\mathcal{G}_t$ be a subset of $\mathcal{M}$ consisting of those
$k$-colorings for which any necklace does not have a fair $q$-splitting using at most $t$ axis-aligned hyperplane cuts. Clearly we have
$$\mathcal{G}_t=\mathcal{M}\setminus\bigcup_{n=1}^{\infty}\mathcal{B}_n.$$

Our aim is to show that sets $\mathcal{B}_n$ are nowhere dense, provided suitable relation between $d,k,q$ and $t$ holds. This will imply that the union $\bigcup_{n=1}^{\infty}\mathcal{B}_n$ is of first category, so the set of `good' colorings $\mathcal{G}_t$, which we are looking for, is dense.

\begin{lemma}\label{LemmaClosed} Sets $\mathcal{B}_n$ are closed in $\mathcal{M}$.
\end{lemma}
\begin{proof} Since $\mathcal{B}_n$ is a finite union of $\mathcal{B}_n^{(t_i)}$ over all $t_1+\dots+t_d\leq t$ it is enough to show that sets $\mathcal{B}_n^{(t_i)}$ are closed. Let $(f_m)_{m\in\N}$ be a sequence of colorings from $\mathcal{B}_n^{(t_i)}$ converging  to $f$ in $\mathcal{M}$. For each $f_m$ there exists a necklace $C_m$ in $[-n,n]^d$ and a fair $q$-splitting of it with granularity at least $1/n$. Since $[-n,n]^d$ is compact we can assume that cubes $C_m$ converge to some cube $C$, and that hyperplane cuts also converge (notice that the granularity of this division of $C$ is at least $1/n$, in particular $C$ is a nontrivial). The only thing left is to define a $q$-labeling of pieces of the division to get $q$ equal parts. But there is a finite number of such labelings for each $C_m$, so one of them appears infinitely many times in the sequence. It is easy to see that this gives a fair $q$-splitting of the necklace $C$, so $f\in\mathcal{B}_n^{(t_i)}$.  
\end{proof}

\begin{lemma}\label{LemmaMain} 
If $k(q-1)>t+d+q-1$, then every $\mathcal{B}_n$ has empty interior. For $d=1$ the same holds provided $k(q-1)>t+2$. 
\end{lemma}
\begin{proof}
Suppose the assertion of the lemma is false: there is some $f\in\mathcal{B}_n$ and $\epsilon>0$ for which $K(f,\epsilon)=\{g\in\mathcal{M}:d(f,g)<\epsilon\}\subset\mathcal{B}_n$. By Lemma \ref{aprox} there is a cube coloring $g\in\mathcal{I}_n$ such that $d(f,g)<\epsilon/2$, so $K(g,\epsilon/2)\subset\mathcal{B}_n$. 

The idea is to modify slightly the coloring $g$ so that the new coloring will be still close to g, but there will be no necklaces inside $[-n,n]^d$ possessing a $q$-splitting of size at most $t$ and granularity at least $1/n$. 

Without loss of generality we may assume that there are equally spaced axis-aligned cubes $C_{i_1,\dots,i_d}$ for $i_1,\dots,i_d\in\{1,\dots,N\}$ in $[-n,n]^d$ such that $N > 4n^2$, and each cube is filled with a unique color in the cube coloring $g$. Notice that edges of cubes are rational, equal to $\frac{2n}{N}$. Let $\delta > 0$ be a rational number satisfying $\delta^d<\min\{\frac{\epsilon}{2N^d},\left(\frac{2n}{N}\right)^d\}$. Inside each cube $C_{i_1,\dots,i_d}$ choose a volume $\delta^d$ axis-aligned cube $D_{i_1,\dots,i_d}$ with corners in rational points $\Q^d$.  Let us call `white' the first color from our set of colors $\{1,\dots,k\}$. For each cube $D_{i_1,\dots,i_d}$ and a color $j=2,\dots,k$ we choose small enough real number $x_{j;i_1,\dots,i_d}>0$ and an axis-aligned cube $D_{j;i_1,\dots,i_d}\subset D_{i_1,\dots,i_d}$ such that:
\begin{enumerate}
\item length of the side of $D_{j;i_1,\dots,i_d}$ equals to $x_{j;i_1,\dots,i_d}$
\item cubes $D_{j;i_1,\dots,i_d}$ for $j=2,\dots,k$ are disjoint
\item one of corners of $D_{j;i_1,\dots,i_d}$ is rational (lies in $\Q^d$)
%\item each axis-aligned hyperplane crosses at most one of $D^j_{i_1,\dots,i_d}$ cubes
\item the set $\{x_{j;i_1,\dots,i_d}:j\in\{2,\dots,k\},i_1,\dots,i_d\in\{1,\dots,N\}\}$ is algebraically independent over $\Q$.
\end{enumerate}

Now we define a coloring $h$, which is a modification of $g$. Let $h$ equals to $g$ outside cubes $D_{i_1,\dots,i_d}$ for $i_1,\dots,i_d\in\{1,\dots,N\}$. Inside the cube $D_{i_1,\dots,i_d}$ let $h$ be equal to `white' everywhere except cubes $D_{j;i_1,\dots,i_d}$, where it is $j$. 

Clearly $d(g,h)<\epsilon/2$, so $h\in\mathcal{B}_n$. Hence there is a necklace $A=[\alpha_1,\alpha_1+\alpha_0]\times\dots\times[\alpha_d,\alpha_d+\alpha_0]\subset [-n,n]^d$ which has a fair $q$-splitting of size at most $t$ and granularity at least $1/n$. Denote by $\beta_1,\dots,\beta_t$ the coordinates of axis-aligned hyperplane cuts (for $H=\{x_i=\beta\}$ we get $\beta$). 

Since granularity of the splitting is at least $1/n$, and edges of cubes $C_{i_1,\dots,i_d}$ are smaller or equal to $1/(2n)$, each piece $P$ of the division contains a cube $C_{i_1,\dots,i_d}\subset P$ for some $i_1,\dots,i_d$ (so in particular $C_{i_1,\dots,i_d}$ is not divided by any hyperplane cut). Each of $q$ parts contains at least one piece and as a consequence at least one cube. For each part we fix such a cube $C_{i^l_1,\dots,i^l_d}$ for $l=1,\dots,q$.

Now the key observation is that the amount of color $j$ different from `white' contained in each part $l$ of the $q$-splitting is a polynomial $W_{j,l}$ from $$\Q[\alpha_i,\beta_m,x_{j;i_1,\dots,i_d}:\;i=0,\dots,d;\;m=1,\dots,t;\;(i_1,\dots,i_d)\in\{1,\dots,N\}^d].$$ 
Comparing the amount of color $j$ in the $l$-th part and the first one we get
$$W_{j,l}(\alpha_i,\beta_m,x_{j;i_1,\dots,i_d})=W_{j,1}(\alpha_i,\beta_m,x_{j;i_1,\dots,i_d}).$$
Since the cube $C_{i^l_1,\dots,i^l_d}$ is not divided by any hyperplane from our splitting, a variable $x_{j;i^l_1,\dots,i^l_d}$ appears only in the component $x^d_{j;i^l_1,\dots,i^l_d}$ in the polynomial $W_{j,l}$. Denote $W_{j,l}'=W_{j,l}-x^d_{j;i^l_1,\dots,i^l_d}$. We get polynomial equations $(\clubsuit)$:
$$(W_{j,1}-W_{j,l}')(\alpha_i,\beta_m,x_{j;i_1,\dots,i_d}:\; (i_1,\dots)\neq (i^r_1,\dots)\; r=2,\dots,q)=x^d_{j;i^l_1,\dots,i^l_d},$$
with coefficients in $\Q$, for all $j=2,\dots,k$ and $l=2,\dots,q$.

Denote by $T$ the set $\{\alpha_i,\beta_m:\;i=0,\dots,d;\;m=1,\dots,t\}$, and suppose its transcendence degree over $\Q$ equals to $t+d+1-c$, for some $c$. We will show that $c\geq 1$, and $c\geq q-1$ for $d=1$ (notice that in the case $d=1$ we can assume, without loss of generality, that $t\geq q-1$). Observe that the volume of each part of the splitting is a polynomial over $\Q$ in variables $\alpha_0,\dots\alpha_d,\beta_1,\dots,\beta_t$. All $q$ parts have to be of the same volume, so we get $q-1$ nontrivial polynomial equations. Each of them shows that $T$ is algebraically dependent set over $\Q$, hence $c\geq 1$. For us the ideal situation would be if equations were `independent' leading to $c=q-1$, since it would give a better lower bound $k(q-1)>t+d+1$. Unfortunately, it is not always the case, see Example \ref{q-1}. However in dimension one these equations are linear, and it is easy to see that they are linearly independent. Hence $c\geq q-1$ for $d=1$. Denote by $\K$ the field
$$\Q(x_{j;i_1,\dots,i_d}:\;(i_1,\dots,i_d)\neq (i^r_1,\dots,i^r_d)\; r=2,\dots,q).$$
Clearly, inclusion $\Q\subset\K$ implies that $\trdeg_{\K}(T)\leq\trdeg_{\Q}(T)=t+d+1-c$. 

Denote by $L$ the set of left sides of equations $(\clubsuit)$, and by $R$ the set of right sides of equations $(\clubsuit)$. We have that $L\subset\K(T)$, so $\trdeg_{\K}(L)\leq \trdeg_{\K}(T)=t+d+1-c$. From condition $(4)$ of the choice of $x_{j;i_1,\dots,i_d}$ it follows that the set $R$ is algebraically independent in $\K$, that is $\trdeg_{\K}(R)=(k-1)(q-1)$. But obviously $L=R$, hence $(k-1)(q-1)\leq t+d$ (and $(k-1)(q-1)\leq t-q+3$ for $d=1$) and we get a contradiction.
\end{proof}

\begin{example}\label{q-1}
Consider $d=2,q=4,t_1=1,t_2=1$. Then equality of volumes of three parts of the splitting implies that the fourth one also has the same volume. This phenomena is caused by the geometry of $\R^2$ (or $\R^d$ for $d\geq 2$ in general) and axis-aligned hyperplanes. It shows that in general it is easier to $q$-split fairly a whole necklace, than its Lebesgue measurable subset. In dimension one these situations do not happen.
\end{example}

Now from Lemmas \ref{LemmaClosed} and \ref{LemmaMain} it follows that if $k(q-1)>t+d+q-1$, or $k(q-1)>t+2$ for $d=1$, then sets $\mathcal{B}_n$ are nowhere dense, which proves the following theorem.

\begin{theorem}\label{TheoremNecklaces}
For every integers $d,k,t\geq 1$ and $q\geq 2$, if $k(q-1)>t+d+q-1$, then there exists a measurable $k$-coloring of $\R^d$ such that no $d$-dimensional necklace has a fair $q$-splitting using at most $t$ axis-aligned hyperplane cuts. The set of such colorings is dense. For $d=1$ the same holds provided $k(q-1)>t+2$. 
\end{theorem}

By modifying slightly the proof of Theorem \ref{TheoremNecklaces} one can easily obtain a version for cuboids with inequality $k(q-1)>t+2d+q-2$.

The idea standing behind the proof is to count `degrees of freedom' of moving a necklace and cuts splitting it and to compare it with the number of equations forced by existence of a fair splitting. If the number of degrees of freedom is less than the number of equations, then there exists a coloring with the property that no necklace can be fairly split. Even though the intuitive meaning of degree of freedom is clear, it is hard to define it rigorously to be able to make use of it. It turns out that the transcendence degree over suitably chosen field of the set $\{\alpha_i,\beta_m:\;i=0,\dots,d;\;m=1,\dots,t\}$ works. 

As we will see in Section \ref{SectionArbitraryHyperplaneCuts}, if we allow arbitrary hyperplane cuts, then each such cut adds $d$ degrees of freedom, not only one as an axis-aligned hyperplane cut. Therefore for arbitrary hyperplane cuts we will get roughly $d$ times worse bound. The same reason explains the difference between results for cubes and cuboids, the latter have larger degree of freedom.

For $d=1$ and $q=2$ Theorem \ref{TheoremNecklaces} gives exactly Theorem \ref{TheoremAGLM} -- the main result of \cite{AlGrLaMi09}. However, the technique of showing that sets $\mathcal{B}_n$ have empty interior is quite different. To show a contradiction we used in \cite{AlGrLaMi09} only one color, which was different from colors of cut points and the end points of an interval. If $d>1$ or $q>2$ this is not always possible. Additionally argument from \cite{AlGrLaMi09} was not only algebraic, but contained also an analytic part, it used some inequalities on measure of colors. We conjecture that our bound from Theorem \ref{TheoremNecklaces} for $q=2$ is tight. 

\begin{conjecture}
For every $d,k,t\geq 1$ and a measurable $k$-coloring of $\R^d$, if $k\leq t+d+1$, then there exists a $d$-dimensional necklace which has a fair $2$-splitting using at most $t$ axis-aligned hyperplane cuts.
\end{conjecture}

However, we know that if we replace a $k$-coloring by arbitrary $k$ continuous measures, then the answer is negative. In Section \ref{SectionColoringsDistinguishingCubes} we provide other examples when the case of arbitrary continuous measures differs from the case of measurable coloring.

The last part of this Section is devoted to colorings of a given fixed necklace that avoid a fair splitting of a certain size.

For $d,k\geq 1$, $q\geq 2$ and a $d$-dimensional necklace $C=[\alpha_1,\alpha_1+\alpha_0]\times\dots\times[\alpha_d,\alpha_d+\alpha_0]$ we define the set $\mathcal{B}_n^{(t_i)}$ of $k$-colorings from $\mathcal{M}$ for which there is a fair $q$-splitting of $C$ with exactly $t_i$ hyperplanes aligned to $i$-th axis, for each $i$, and granularity at least $1/n$. Finally let
$$\mathcal{B}_n=\bigcup_{t_1+\dots+t_d\leq t}\mathcal{B}_n^{(t_i)}$$
be `bad' colorings. The aim, as previously, is to show that sets $\mathcal{B}_n$ are nowhere dense, provided suitable relation between $d,k,q$ and $t$ holds. This will imply that the union $\bigcup_{n=1}^{\infty}\mathcal{B}_n$ is of first category, so the set of colorings which we are looking for is dense. Analogously to Lemma \ref{LemmaClosed} sets $\mathcal{B}_n$ are closed in $\mathcal{M}$. Similarly to Lemma \ref{LemmaMain} if $k(q-1)>t+q-2$, then $\mathcal{B}_n$ has empty interior. The only difference is that we want the set 
$$\{x_{j;i_1,\dots,i_d}:j\in\{2,\dots,k\},i_1,\dots,i_d\in\{1,\dots,N\}\}$$ 
to be algebraically independent over $\Q(\alpha_i:i=0,\dots,d)$. We also add $\alpha_0,\dots,\alpha_d$ to the base field $\K$, so the transcendence degree of the left side of equations is equal to $t-c\leq t-1$, while of the right side it is still $(k-1)(q-1)$. This proves the following theorem.

\begin{theorem}\label{TheoremNecklace}
For every integers $d,k,t\geq 1$ and $q\geq 2$, if $k(q-1)>t+q-2$, then there is a measurable $k$-coloring of a $d$-dimensional necklace for which there is no fair $q$-splitting using at most $t$ axis-aligned hyperplane cuts. The set of such colorings is dense. 
\end{theorem}

For $q=2$ the above theorem is tight. It asserts that there are measurable $(t+1)$-colorings of a given necklace for which there is no fair $2$-splitting using at most $t$ axis-aligned hyperplane cuts. The set of such colorings is even dense. Due to a result of Goldberg and West \cite{GoWe85} it is the minimum number of colors. Together with Theorem \ref{lozi} of de Longueville and \v{Z}ivaljevi\'{c} this fully characterizes the set of tuples $(t_1,\dots,t_d)\in\N^d$, such that any $k$-colored $d$-dimensional necklace has a fair $2$-splitting using at most $t_i$ hyperplanes aligned to $i$-th axis for each $i$. Namely, it is the set $$\{(t_1,\dots,t_d)\in\N^d:\;k\leq t_1+\dots+t_d\}.$$

%%%%%%%%%%%%%%%%%%%%%%%%%%%%%%%%%%%%%%%%%%%%%%%%%%%%%%%%%%%%%%%%%%%%%%%%%%%%%%%%%%%%%%%%%%%%%%%%%%%%%%%%%%%%%%%%%%%%%%%%%%%%%%%%%%%%%%%%%%%%%%%%%%%%%%%%%%%%%%%%%%%%%%%%%%
\section{Arbitrary Hyperplane Cuts}\label{SectionArbitraryHyperplaneCuts}
%%%%%%%%%%%%%%%%%%%%%%%%%%%%%%%%%%%%%%%%%%%%%%%%%%%%%%%%%%%%%%%%%%%%%%%%%%%%%%%%%%%%%%%%%%%%%%%%%%%%%%%%%%%%%%%%%%%%%%%%%%%%%%%%%%%%%%%%%%%%%%%%%%%%%%%%%%%%%%%%%%%%%%%%%%

An arbitrary hyperplane (not necessarily axis-aligned) can be described by $d$ numbers $\beta^1,\dots,\beta^d$, which we will call \emph{parameters}. There are several ways of defining parameters, fortunately usually transition functions between different definitions are rational over $\Q$. For example for a hyperplane not passing through $O=(0,\dots,0)$ we can define parameters to be coefficients of its normalized equation $\beta^1x_1+\dots+\beta^dx_d=1$. In general pick $d+1$ rational points in general position $P_1,\dots,P_{d+1}$. Then no hyperplane contains all of them. We describe a hyperplane $H$ using coefficients of its normalized equation with $O$ being moved to a point $P_i$ which does not belong to $H$. We will make use of the following classical fact.

\begin{lemma}\label{rat} 
Volume of a bounded set in $\R^d$ which is an intersection of a finite number of half spaces is a rational function over $\Q$ of parameters of hyperplanes supporting these half spaces.
\end{lemma}

\begin{proof}
Such a set $S$ is a convex polytope. From Cramer's formula it follows that coordinates of vertices of $S$ are rational functions over $\Q$ of parameters. We make the barycentric subdivision of $S$ into simplices. Vertices of these simplices are barycenters of faces of $S$, hence they are also rational functions over $\Q$ of parameters.

Volume of each simplex equals to the determinant of the matrix of coordinates of its vertices divided by $d!$. Thus it is a rational function over $\Q$ of parameters. Then so is the volume of the original polytope.   
\end{proof}

Granularity of the splitting is the largest $g$ such that any piece of the splitting contains an axis-aligned cube with side of length $g$. We denote by $\mathcal{B}_n$ the set of $k$-colorings from $\mathcal{M}$ for which there exists at least one $d$-dimensional necklace contained in $[-n, n]^d$ which has a fair $q$-splitting with at most $t$ hyperplanes and granularity at least $1/n$. These are `bad' colorings. 

As in the previous Section \ref{SectionAxisAlignedHyperplaneCuts} our aim is to show that sets $\mathcal{B}_n$ are nowhere dense, provided suitable relation between $d,k,q$ and $t$ holds. This will imply that the set of colorings which we are looking for is dense.

The proof that sets $\mathcal{B}_n$ are closed in $\mathcal{M}$ goes exactly the same as the proof of Lemma \ref{LemmaClosed}. To show that if $(k-1)(q-1)>dt+d$, then every $\mathcal{B}_n$ has empty interior we repeat the proof of Lemma \ref{LemmaMain} and make small modifications. Instead of a single parameter $\beta_i$ of a hyperplane we have $d$ parameters $\beta_i^1,\dots,\beta_i^d$. Due to Lemma \ref{rat} the amount of color $j$ different from `white' contained in each part $l$ of the $q$-splitting is a rational function $R_{j,l}$ from 
$$\Q(\alpha_i,\beta_m^p,x_{j;i_1,\dots,i_d}:\;i=0,\dots,d;\;m=1,\dots,t;\;p=1,\dots,d;\;(i_1,\dots,i_d)),$$ 
instead of a polynomial $W_{j,l}$. Now the transcendence degree over the field
$$\K=\Q(x_{j;i_1,\dots,i_d}:\;(i_1,\dots,i_d)\neq (i^r_1,\dots,i^r_d)\; r=2,\dots,q)$$ 
of the set of left sides of equations analogous to $(\clubsuit)$ is at most $dt+d+1-1$, while of the set of right sides equals to $(k-1)(q-1)$. Hence the hypothesis that $\mathcal{B}_n$ has non-empty interior implies inequality $(k-1)(q-1)\leq dt+d$. It follows that if $(k-1)(q-1)>dt+d$, then sets $\mathcal{B}_n$ are nowhere dense. This proves the following theorem.

\begin{theorem}\label{TheoremNecklacesArbitrary}
For every integers $d,k,t\geq 1$ and $q\geq 2$, if $k(q-1)>dt+d+q-1$, then there is a measurable $k$-coloring of $\R^d$ such that no $d$-dimensional necklace has a fair $q$-splitting using at most $t$ arbitrary hyperplane cuts. The set of such colorings is dense. 
\end{theorem}

Modifying slightly the above argument one may obtain a version for cuboids with inequality $k(q-1)>dt+2d+q-2$. Other modification, similar to the proof of Theorem \ref{TheoremNecklace}, leads to the following theorem.

\begin{theorem}\label{TheoremNecklaceArbitrary}
For every integers $d,k,t\geq 1$ and $q\geq 2$, if $k(q-1)>dt+q-2$, then there is a measurable $k$-coloring of a $d$-dimensional necklace for which there is no fair $q$-splitting using at most $t$ arbitrary hyperplane cuts. The set of such colorings is dense. 
\end{theorem}

For $q=2$ we get a necessary condition, that is $k(q-1)\leq dt$, for mass equipartitioning problem. It is an interesting question if this condition is also sufficient for existence of a fair $q$-splitting of size at most $t$ of a $k$-colored $d$-dimensional necklace. For $(d,k,q,t)=(d,k,2^t,t)$ the question coincides with the conjecture of Ramos \cite{Ra96}. A small support is the positive answer for $(d,k,q,t)=(d,d,2^l,2^l-1)$, which follows directly from $2^l-1$ applications of the Ham Sandwich Theorem for measures \cite{Ma03}. 

%%%%%%%%%%%%%%%%%%%%%%%%%%%%%%%%%%%%%%%%%%%%%%%%%%%%%%%%%%%%%%%%%%%%%%%%%%%%%%%%%%%%%%%%%%%%%%%%%%%%%%%%%%%%%%%%%%%%%%%%%%%%%%%%%%%%%%%%%%%%%%%%%%%%%%%%%%%%%%%%%%%%%%%%%%
\section{Colorings Distinguishing Cubes}\label{SectionColoringsDistinguishingCubes}
%%%%%%%%%%%%%%%%%%%%%%%%%%%%%%%%%%%%%%%%%%%%%%%%%%%%%%%%%%%%%%%%%%%%%%%%%%%%%%%%%%%%%%%%%%%%%%%%%%%%%%%%%%%%%%%%%%%%%%%%%%%%%%%%%%%%%%%%%%%%%%%%%%%%%%%%%%%%%%%%%%%%%%%%%%

Here we are interested in a problem of a slightly different flavour. We say that a measurable coloring of $\R^d$ \emph{distinguishes axis-aligned cubes} if no two nontrivial axis-aligned cubes contain the same measure of every color. What is the minimum number of colors needed for such a coloring?

We get our results in a similar way to the proof of Theorem \ref{TheoremNecklaces}. For each $n$ we define a set $\mathcal{B}_n$ of all $k$-colorings from $\mathcal{M}$ for which there exist two $d$-dimensional cubes $A,B$ contained in $[-n, n]^d$ which have the same measure of every color, and such that $A\setminus B$ contains a translation of the cube $\left(0,\frac{1}{n}\right)^d$. The `bad' colorings are $\mathcal{B}=\bigcup_{n=1}^{\infty}\mathcal{B}_n$. Sets $\mathcal{B}_n$ are closed, and we can easily get an analog of Lemma \ref{LemmaMain}. This shows the following theorem.

\begin{theorem}\label{2d+3}
For every $d\geq 1$ there exists a measurable $(2d+3)$-coloring of $\R^d$ distinguishing axis-aligned cubes. The set of such colorings is dense. 
\end{theorem}

For $d=1$ this gives the second result of \cite{AlGrLaMi09}, namely the existence of a measurable $5$-coloring of $\R$ with no two equally colored intervals. We suspect that this number of colors is optimal. 

\begin{conjecture}
For every $d\geq 1$ and for every measurable $(2d+2)$-coloring of $\R^d$ there are two nontrivial axis-aligned $d$-dimensional cubes which have the same measure of every color. 
\end{conjecture}

A slight modification of the argument leads to a version of Theorem \ref{2d+3} for cuboids with $4d+1$ colors.

Vre\'{c}ica and \v{Z}ivaljevi\'{c} \cite{VrZi13} solved the case of arbitrary continuous probability measures instead of measurable colorings of $\R^d$. They prove that $d+1$ measures are sufficient to distinguish axis-aligned cubes, and that it is tight.

\begin{theorem}[Vre\'{c}ica, \v{Z}ivaljevi\'{c}, \cite{VrZi13}]\label{vrec}
For every $d$ continuous probability measures $\mu_1,\dots,\mu_d$ on $\R^d$ there are two nontrivial axis-aligned cubes $C$ and $C'$ such that $\mu_i(C)=\mu_i(C')$ holds for every $i$. Moreover, there exist $d+1$ continuous probability measures on $\R^d$ distinguishing axis-aligned cubes.
\end{theorem}

In fact Vre\'{c}ica and \v{Z}ivaljevi\'{c} show that even one can find any finite number of axis-aligned cubes in $\R^d$ which are not distinguished by given $d$ measures. They also get results analogous to Theorem \ref{vrec} for cuboids in $\R^d$ with $2d-1$ colors. 

For sure, the minimum number of colors in the two cases (arbitrary continuous measures, and measurable colorings of $\R^d$) is different. One can show the following. 

\begin{example} 
There are two measures on $\R$ distinguishing intervals. For example, the ones given by integrals of functions $1$ and $e^x$. 

On the other hand no measurable $2$-coloring of $\R$ distinguishes intervals. It is a consequence of properties of a measurable set (one can find an interval with large or small ratio of the set), and continuity of Lebesgue measure with respect to changes of interval endpoints.
\end{example}

%%%%%%%%%%%%%%%%%%%%%%%%%%%%%%%%%%%%%%%%%%%%%%%%%%%%%%%%%%%%%%%%%%%%%%%%%%%%%%%%%%%%%%%%%%%%%%%%%%%%%%%%%%%%%%%%%%%%%%%%%%%%%%%%%%%%%%%%%%%%%%%%%%%%%%%%%%%%%%%%%%%%%%%%%%
\section{A discrete case}\label{SectionDiscreteCase}
%%%%%%%%%%%%%%%%%%%%%%%%%%%%%%%%%%%%%%%%%%%%%%%%%%%%%%%%%%%%%%%%%%%%%%%%%%%%%%%%%%%%%%%%%%%%%%%%%%%%%%%%%%%%%%%%%%%%%%%%%%%%%%%%%%%%%%%%%%%%%%%%%%%%%%%%%%%%%%%%%%%%%%%%%%

In dimension one a continuous version of the necklace splitting problem was a tool to derive a discrete one, and the minimal number of cuts needed is the same. For higher dimensions we will see that there is no such correspondence. 

For a fixed integer $q\geq 2$ a \emph{$k$-colored discrete $d$-dimensional necklace} is a cube in $\N^d$ partitioned into $k$ sets, each of cardinality divisible by $q$. A \emph{fair $q$-splitting} of size $t$ is a division using $t$ axis-aligned hyperplane cuts into cuboids, which can be partitioned into $q$ parts, each capturing exactly $1/q$ of the total number of points of each color. Alon's necklace splitting theorem \cite{Al87} states that every $k$-colored discrete $1$-dimensional necklace has a fair $q$-splitting using at most $k(q-1)$ cuts. Grytczuk and Lubawski \cite{GrLu12} asked what is the minimum number of axis-aligned hyperplane cuts that are needed to $2$-split fairly any $k$-colored discrete $d$-dimensional necklace. At first glance, it is not clear why this number is finite. We give a partial answer to the more general question with $q$ parts by showing an upper and a lower bound which differ by multiplicative factor $4$. 

\begin{theorem}\label{discrete}
For every $d,k\geq 1$ and $q\geq 2$, the minimum number of axis-aligned hyperplane cuts $t$ that are needed to $q$-split fairly any $k$-colored discrete $d$-dimensional necklace satisfies inequalities $d(k-1)\frac{q}{2}\leq t\leq (2d-1)k(q-1)$.
\end{theorem}

\begin{proof}
The bound from above. Let $N$ be a $k$-colored discrete $d$-dimensional necklace. Let $<$ be the lexicographic order on $N$, that is $(x_1,\dots,x_d)<(y_1,\dots,y_d)$ if and only if $x_i=y_i$ for $i<j$ and $x_j<y_j$. We can imagine discrete $d$-dimensional necklace stretched on a line according to the linear order $<$. It forms a discrete $1$-dimensional necklace $N'$. From Alon's necklace splitting theorem \cite{Al87} we get that there is a fair $q$-splitting of $N'$ using at most $k(q-1)$ cuts. Suppose there is a cut between $x<y$, which are consecutive elements in our order. Suppose $x_i=y_i$ for $i<j$ and $x_j<y_j$. We can realize this cut with $2j-1\leq 2d-1$ axis-aligned hyperplane cuts in $\N^d$. Namely by cutting with hyperplanes $z_1=x_1-\frac{1}{2}, z_1=x_1+\frac{1}{2},\dots,z_{j-1}=x_{j-1}-\frac{1}{2},z_{j-1}=x_{j-1}+\frac{1}{2}$, and $z_j=\frac{x_j+y_j}{2}$. Hence $(2d-1)k(q-1)$ hyperplane cuts suffice.

The bound from below. First consider only one color, we are going to show that in general we need at least $d\frac{q}{2}$ cuts to split it. We will make a double counting argument. Pick any subset $N$ of a cube $C=\{1,\dots,n\}^d$, there are $2^{n^d}$ possibilities. Now $t$ cuts can be chosen in at most $(dn)^t$ ways. They divide the cube $C$ into at most $(t+1)^d$ cuboids, so there are at most $q^{(t+1)^d}$ possibilities of splitting these cuboids into $q$ parts. A fixed set of cuts and a labeling of pieces gives a $q$-splitting of $C$ into parts $A_1,\dots,A_q$, say of cardinalities $a_1,\dots,a_q$. We have $a_1+\dots+a_q=n^d$. Let us estimate the number of sets $N$ for which it is a fair $q$-splitting. It is equal to the sum $$\sum_i\binom{a_1}{i}\cdots\binom{a_q}{i},$$ since all parts have to contain the same number of points. It is easy to see that this sum is maximized for $a_i=\left\lceil\frac{1}{q}n^d\right\rceil$ or $\left\lfloor\frac{1}{q}n^d\right\rfloor$. Then from the Cauchy-Schwarz inequality: 
$$\sum_{i}\binom{\frac{1}{q}n^d}{i}^q\leq\left(\sum_{i}\binom{\frac{1}{q}n^d}{i}^2\right)^{\frac{q}{2}}=\binom{\frac{2}{q}n^d}{\frac{1}{q}n^d}^{\frac{q}{2}}=\Theta\left(\frac{2^{n^d}}{n^{d\frac{q}{2}}}\right).$$ %\thicksim
Hence since we assume that $t$ cuts suffice, we have an inequality:
$$\Theta\left((dn)^tq^{(t+1)^d}\frac{2^{n^d}}{n^{d\frac{q}{2}}}\right)\geq 2^{n^d}.$$
We get that for $n_0$ big enough there is a set $N_0\subset\{1,\dots,n_0\}^d$ for which at least $t\geq d\frac{q}{2}$ cuts are needed. Now when there are $k$ colors, we take $k-1$ copies of $N_0$, place them in $\N^d$ inside cubes $\{1,\dots,n_0\}^d,\dots,((k-2)n_0,\dots,(k-2)n_0)+\{1,\dots,n_0\}^d$, and color using colors $2,\dots,k$ respectively. The remaining part of a big cube $\{1,\dots,(k-1)n_0\}^d$ we color with $1$. Each axis-aligned hyperplane cuts only one of colors $2,\dots,k$, and for each of them we need at least $d\frac{q}{2}$ cuts, so in total at least $d(k-1)\frac{q}{2}$ cuts.
\end{proof}

The gap between an upper and a lower bound from Theorem \ref{discrete} motivates the following generalization of the question of Grytczuk and Lubawski \cite{GrLu12}.

\begin{question}
What is the least possible number $t$, such that any $k$-colored discrete $d$-dimensional necklace has a  fair $q$-splitting using at most $t$ axis aligned hyperplane cuts?
\end{question}

In our opinion the correct answer should be around the value of an upper bound $(2d-1)k(q-1)$. For $d=q=2$ this almost matches with a lower bound $t\geq 3k-2$ implied by a construction due to Petecki cf. \cite{GrLu12}. 

%%%%%%%%%%%%%%%%%%%%%%%%%%%%%%%%%%%%%%%%%%%%%%%%%%%%%%%%%%%%%%%%%%%%%%%%%%%%%%%%%%%%%%%%%%%%%%%%%%%%%%%%%%%%%%%%%%%%%%%%%%%%%%%%%%%%%%%%%%%%%%%%%%%%%%%%%%%%%%%%%%%%%%%%%%
\section*{Acknowledgements}
%%%%%%%%%%%%%%%%%%%%%%%%%%%%%%%%%%%%%%%%%%%%%%%%%%%%%%%%%%%%%%%%%%%%%%%%%%%%%%%%%%%%%%%%%%%%%%%%%%%%%%%%%%%%%%%%%%%%%%%%%%%%%%%%%%%%%%%%%%%%%%%%%%%%%%%%%%%%%%%%%%%%%%%%%%

I would like to thank Piotr Micek and Wojciech Lubawski for the help in preparation of this manuscript. 

%%%%%%%%%%%%%%%%%%%%%%%%%%%%%%%%%%%%%%%%%%%%%%%%%%%%%%%%%%%%%%%%%%%%%%%%%%%%%%%%%%%%%%%%%%%%%%%%%%%%%%%%%%%%%%%%%%%%%%%%%%%%%%%%%%%%%%%%%%%%%%%%%%%%%%%%%%%%%%%%%%%%%%%%%%


\begin{thebibliography}{99}
%%%%%%%%%%%%%%%%%%%%%%%%%%%%%%%%%%%%%%%%%%%%%%%%%%%%%%%%%%%%%%%%%%%%%%%%%%%%%%%%%%%%%%%%%%%%%%%%%%%%%%%%%%%%%%%%%%%%%%%%%%%%%%%%%%%%%%%%%%%%%%%%%%%%%%%%%%%%%%%%%%%%%%%%%%

\bibitem{AlWe86} N. Alon, D. West, The Borsuk-Ulam theorem and bisection of necklaces, Proc. Amer. Math.
Soc. 98 (1986), 623-628. 

\bibitem{Al87} N. Alon, Splitting necklaces, Advances in Math. 63 (1987), 247-253. 

\bibitem{AlGrLaMi09} N. Alon, J. Grytczuk, M. Laso\'{n}, M. Micha\l ek, Splitting necklaces and measurable  colorings of the real line, Proc. Amer. Math. Soc. 137 (2009), no. 5, 1593-1599. 

\bibitem{BlZi07} P. Blagojevi\'{c}, G. Ziegler, The ideal-valued index for a dihedral group action, and mass partition by two hyperplanes, Topology Appl. 158 (2011), no. 12, 1326-1351.

\bibitem{Ed87} H. Edelsbrunner, Algorithms in Combinatorial Geometry, volume 10 of EATCS Monographs in Theoretical Computer Science, Springer-Verlag, Berlin, 1987.

\bibitem{GoWe85} C. Goldberg, D. West, Bisection of circle colorings, SIAM J. Algebraic Discrete Methods
6 (1985), no. 1, 93-106. 

\bibitem{GrLu12} J. Grytczuk, W. Lubawski, Splitting multidimensional necklaces and measurable colorings of Euclidean spaces, to appear in SIAM J. Discrete Math., arXiv:1209.1809v1.

\bibitem{GrSl03} J. Grytczuk, W. \'{S}liwa, Non-repetitive colorings of infinite sets, Discrete Math. 265 (2003), no. 1-3, 365-373.

\bibitem{HoRi65} C. Hobby, J. Rice, A moment problem in L1 approximation, Proc. Amer. Math. Soc. 16 (1965), 665-670.

\bibitem{LoZi08} M. de Longueville, R. \v{Z}ivaljevi\'{c}, Splitting multidimensional necklaces, Advances in Math. 218 (2008), no. 3, 926-939.

\bibitem{MaVrZi06} P. Mani-Levitska, S. Vre\'{c}ica, R. \v{Z}ivaljevi\'{c}, Topology and combinatorics of partitions of masses by hyperplanes, Advances in Math. 207 (2006), no. 1, 266-296.

\bibitem{Ma03} J. Matou\v{s}ek, Using the Borsuk-Ulam theorem: Lectures on Topological Methods in Combinatorics and Geometry, Springer-Verlag, Berlin, 2003. 

\bibitem{Ma10} B. Matschke, A note on mass partitions by hyperplanes, arXiv 1001.0193.

\bibitem{Ox80} J. Oxtoby, Measure and Category, Graduate Texts in Mathematics 2, Springer-Verlag, New York-Berlin, 1980. 

\bibitem{Ra96} E. Ramos, Equipartition of mass distributions by hyperplanes, Discrete Comput. Geom. 15 (1996), no. 2, 147-167.

\bibitem{VrZi13} S. Vre\'{c}ica, R. \v{Z}ivaljevi\'{c}, Measurable patterns, necklaces, and sets indiscernible by measure, arXiv:1305.7474.

\end{thebibliography}
\end{document}